\documentclass[reqno, 12pt]{amsart} 
\usepackage{array} 
\usepackage{amsmath}
\usepackage{amsfonts}
\usepackage{amssymb}
\usepackage{mathrsfs}
\usepackage{enumerate}
\usepackage{amsthm}
\usepackage{stmaryrd}
\usepackage{verbatim}
\usepackage{amsmath, amscd}
\usepackage[usenames,dvipsnames]{color}
\usepackage{enumitem}
\usepackage{bm}
\usepackage{xy}
\xyoption{all}
\usepackage{color}
\usepackage{marvosym}
\usepackage{amsbsy}
\usepackage{hyperref}
\usepackage{tikz, tikz-cd}
\usepackage{marginnote}
\usetikzlibrary{matrix,arrows,backgrounds}%, decorations, pathmorphing, positioning, fit}
\usepackage[backgroundcolor=white]{todonotes}
\usepackage{hyperref}
\hypersetup{
    colorlinks=true,
    citecolor=purple,
    linkcolor=blue,
    filecolor=magenta,      
    urlcolor=blue,
}
\newtheorem{theorem}{Theorem}[section]

\newtheorem{lemma}[theorem]{Lemma}
\newtheorem{proposition}[theorem]{Proposition}
\newtheorem{corollary}[theorem]{Corollary}

\theoremstyle{definition}
\newtheorem{definition}[theorem]{Definition}
\newtheorem{remark}[theorem]{Remark}

\newtheorem{example}[theorem]{Example}

\newcommand{\op}[1]{\operatorname{#1}}

\newcommand{\modmod}[1]{/ \! \! / \!_{#1}}

\newcommand{\newterm}{\textsf}

\newcommand{\dbcoh}[1]{\operatorname{D}^{\operatorname{b}}(\operatorname{coh }#1)}

\newcommand{\dabs}{\op{D}^{\op{abs}}}
\newcommand{\gm}{\mathbb{G}_m}

\newcommand{\Hom}{\text{Hom}}

\def\ker{\op{\mbox{ker}}}

\def\Z{\op{\mathbb{Z}}}
\def\C{\op{\mathbb{C}}}
\def\R{\op{\mathbb{R}}}
\def\Q{\op{\mathbb{Q}}}

\def\O{\op{\mathcal{O}}}
\def\A{\op{\mathbb{A}}}

\def\Gm{\op{\mathbb{G}_m}}

\def\n{n}

\def\tot{\op{tot}}

\newcommand{\Amp}{\operatorname{Amp}}

\newcommand{\Cone}{\operatorname{Cone}}

\title[Kuznetsov Categories for GLSMs]{Kuznetsov Categories for Gauged Linear Sigma Models}

\author[Favero]{David Favero}
\address{
	\begin{tabular}{l}
		David Favero \\
        \hspace{.1in} University of Minnesota, School of Mathematics \\
		\hspace{.1in} 206 Church Street, Minneapolis MN 55455, USA \\
		\hspace{.1in} Email: {\bf favero@umn.edu} \\
	\end{tabular}
}

\author[Kaplan]{Daniel Kaplan}
\address{
  \begin{tabular}{l}
   Daniel Kaplan  \\
   \hspace{.1in} California State University Long Beach, College of Natural Sciences and Mathematics \\
   \hspace{.1in} 1250 N Bellflower Blvd, Long Beach, CA 90840, USA \\
   \hspace{.1in} Email: {\bf dan.kaplan@csulb.edu} \\
  \end{tabular}
}

\author[Kelly]{Tyler L. Kelly}
\address{
  \begin{tabular}{l}
   Tyler L. Kelly \\
   \hspace{.1in} Queen Mary University of London, School of Mathematical Sciences \\
   \hspace{.1in} London, E1 4NS,  United Kingdom \\
   \hspace{.1in} Email: {\bf t.l.kelly@qmul.ac.uk} \\
  \end{tabular}
}

\numberwithin{equation}{section}

\begin{document}

\begin{abstract}
We define Kuznetsov and anti-Kuznetsov categories for gauged linear sigma models. We show that for complete intersections of ample divisors in smooth projective  toric varieties, the Kuznetsov category is left orthogonal to an exceptional collection. We prove that any complete intersection of $r\ge 2$ ample divisors in a Fano GIT quotient is a Fano visitor and the derived category of its Fano host is equivalent to an anti-Kuznetsov category of a gauged linear sigma model.
\end{abstract}

\maketitle
\setcounter{tocdepth}{1}
%\tableofcontents
\vspace{-1cm}
\section{Introduction}

The derived category of coherent sheaves $\dbcoh{X}$ on a variety $X$ is an `almost-perfect' invariant of $X$, encapsulating much of its geometry and descending to many other invariants, such as cohomology and $K$-theory. When $\dbcoh{X}$ is decomposable, certain admissible subcategories can hold interesting information about $X$. First for  cubic fourfolds followed by other Fano varieties, Kuznetsov pioneered identifying such admissible subcategories as (left) orthogonals to exceptional objects \cite{Kuz05, Kuz09, Kuz10, KuzICM}.  

These components, now called  Kuznetsov (or residual) components of $\dbcoh{X}$, have since enjoyed considerable attention in the literature. For example, they have been used to probe rationality \cite{Kuz10, AT}, study mirror symmetry \cite{SheridanSmith}, explore Bridgeland stability \cite{BLMS}, and establish categorical Torelli theorems \cite{BMMS, HR, PS}.  Despite their range of applications in many examples, there is at present no uniform definition of such categories.

In this paper, we propose a uniform definition of Kuznetsov categories for gauged linear sigma models (GLSMs). Witten introduced gauged linear sigma models in physics as a generalized setting for the correspondence between Landau-Ginzburg models and Calabi-Yau manifolds \cite{W93}.  A GLSM contains the data needed to construct a (well-behaved) superpotential $w: X \rightarrow \A^1$ on a stack $X$ presented as a GIT quotient $X = [V\modmod{\theta} G]$ coming from a representation $V$ of an algebraic group $G$. Given a GLSM, there is a corresponding absolute derived category defined by Positselski \cite{P11} generalizing the category of matrix factorizations. 

Choices of GIT quotients of $V$ by $G$ are parameterized by the characters of $G$. The character $\theta$, determines the Zariski-open, (semi)stability locus $V_{\theta}$, and hence the GIT quotient $[V\modmod{\theta}G]$. Varying the character, one obtains relations between the various absolute derived categories (see, e.g.,\ \cite{HHP,  Segal, HW, DHL, VGIT,  HLS, SvdB}). 
%We identify a distinguished character $\theta_K \in \widehat G$ and relate each GIT quotient . Choices of GIT quotients of $V$ by $G$ are parameterized by the characters of $G$, %and there is the GIT fan in $\widehat{G}\otimes \R$ which gives a wall-and-chamber structure for the GIT quotients. 

We identify a distinguished character $\theta_K \in \widehat G$ and define a \newterm{Kuznetsov chamber} to be a chamber in the GIT fan which contains this distinguished character $\theta_K$ (Definition~\ref{def: Kuznetsov chamber}). We define a \newterm{Kuznetsov category} $\mathcal K$ of the GLSM to be the absolute derived category associated to this new GLSM corresponding to the superpotential $w_{K+\epsilon}: X_{K+\epsilon} \to \A^1$ where $X_{K+\epsilon} = [V\modmod{\theta_{K+\epsilon}} G]$ and $\theta_{K + \epsilon}$ is in the interior of said Kuznetsov chamber (Definition~\ref{def: Kuznetsov category}).  This perspective also has the advantage of an obvious symmetry; one can define an \newterm{anti-Kuznetsov category} $-\mathcal K$ as well using an anti-Kuznetsov chamber containing the character $\theta_{-K} := \theta_K^{-1}$.

The Kuznetsov (resp.\ anti-Kuznetsov) category apriori depends on the choice of chamber.   Nevertheless, it follows from the work of Ballard-Favero-Katzarkov \cite{BFKv2} that all choices of Kuznetsov (resp.\ anti-Kuznetsov) chambers  yield equivalent Kuznetsov (resp.\ anti-Kuznetsov) categories (see  Corollary~\ref{cor: independence}).

We then study Kuznetsov categories for \emph{geometric} GLSMs (see Definition~\ref{def: geometric}). In this case, $X$ is a (stacky) vector bundle $\mathcal {E}$ over some variety $B$ and the superpotential $w$ corresponds to a regular global section $s \in \Gamma(B, \mathcal{E}^\vee)$. The absolute derived category of the geometric GLSM is  equivalent to the derived category $\dbcoh{Z}$ of a complete intersection $Z:=Z(s)$ in $B$. We use this approach to provide new semi-orthogonal decompositions of complete intersections in terms of their Kuznetsov categories. In particular, for geometric abelian GLSMs where $\mathcal{E}$ is a split vector bundle consisting of anti-ample line bundles, we prove that the orthogonal to the Kuznetsov category is an exceptional collection. This has the following geometric reinterpretation:

\begin{theorem}[=Theorem~\ref{thm: main theorem}]\label{thm: intro thm 1}
Let $Z$ be a complete intersection of ample divisors in a smooth projective toric variety (or Deligne-Mumford stack).  Then, there are semiorthogonal decompositions
\begin{equation}\label{eq: intro kuz}
\dbcoh{Z} = \langle \mathcal K, E_1, \dots, E_t \rangle
\end{equation}
and
\begin{equation}\label{eq: intro antikuz}
-\mathcal K = \langle \dbcoh{Z}, E_1, \dots, E_s \rangle
\end{equation}
where the $E_i$ are exceptional objects and $\mathcal K$ (resp.\ $-\mathcal K$) is the Kuznetsov (resp.\ anti-Kuznetsov) category for the geometric GLSM associated to $Z$.
\end{theorem}

%Here, the admissible subcategory $\mathcal K$ is the Kuznetsov category of the (geometric) GLSM associated to the complete intersection $Y$. Moreover, using the GLSM approach, we can also define an anti-Kuznetsov category $-\mathcal K$ corresponding to an anti-Kuznetsov chamber and $\dbcoh{Y}$ is an admissible subcategory of $-\mathcal K$. 

This result can be viewed as a toric analogue of Orlov's theorem ~\cite[Theorem 3.11]{Orlov}. As in Orlov's theorem, more is known when we specialize to the hypersurface case:

\begin{theorem}[=Theorem~\ref{thm: Fano CY category}]
Let $Z$ be a smooth hypersurface in a smooth projective Fano toric variety $X$ (or Deligne-Mumford stack). Assume $\O_X(Z) \cong \omega_X^{-q}$ for some $q\in \Q_{>0}$.  
% If $q \leq 1$, then we have a semi-orthogonal decomposition
%   $$
%   \dbcoh{Z} \cong \langle \mathcal K , E_1, \dots, E_t\rangle,
%   $$
%   and if $q \geq 1$, then we have a semi-orthogonal decomposition
%     $$
%   -\mathcal K  \cong \langle\dbcoh{Z}, E_1, \dots, E_s\rangle,
%   $$
%   where each $E_i$ is an exceptional object.
\begin{enumerate}
\item If $q\leq 1$ then the Kuznetsov category $\mathcal K$ in~\eqref{eq: intro kuz} is fractional  Calabi-Yau. 
\item If $q\geq 1$ then the anti-Kuznetsov category $-\mathcal K$ in~\eqref{eq: intro antikuz} is fractional Calabi-Yau.  
\item If $q = \frac{1}{r}$ for some $r \in \Z$ and $\O_X(Z)$ is Cartier, then $\mathcal K$ in~\eqref{eq: intro kuz} is Calabi-Yau. 
\end{enumerate}
\end{theorem}

In 2011, Bondal asked \cite{BBF}: Which smooth projective varieties $X$ have derived categories that are admissible subcategories of the derived category of some smooth projective Fano variety $Y$? If such an $X$ and $Y$ exist, we call $X$ a \newterm{Fano visitor} and $Y$ a \newterm{Fano host}. Anti-Kuznetsov categories provide good candidates for Fano hosts. 
\begin{theorem}[=Corollary \ref{cor: xmas miracle}]
Any complete intersection of $r\ge 2$ ample divisors in a Fano GIT quotient $X$ is a Fano visitor and the derived category of the Fano host is equivalent to an anti-Kuznetsov category of a GLSM.
\end{theorem}
\noindent This generalizes the main result of \cite{KKLL} which shows that any complete intersection in projective space is a Fano visitor.  % We finish the paper with examples that may invite future study and an example that shows that the ampleness hypothesis in Theorem~\ref{thm: intro thm 1} is necessary.

\subsection*{Acknowledgments}
The first author was supported by the NSF under DMS award numbers 2302262 and 2412039.
The second author was supported by the EPSRC Mathematical Sciences Small Grant EP/Y033574/1. The third author was supported by the UKRI Future Leaders Fellowship MR/T01783X/1, its renewal MR/Y033841/1, and the EPSRC Mathematical Sciences Small Grant EP/Y033574/1. 

\section{Kuznetsov Categories for GLSMs}

In this section, we define the Kuznetsov category of a GLSM.  Roughly, we take a chamber of the corresponding GIT problem corresponding to the canonical class and take its corresponding absolute derived category (see \cite{P11} and \cite[Definition 3.9]{BFK14a}).  We expand on these notions below.

The precise mathematical definition of a GLSM consists of a 5-tuple satisfying several conditions. Our definition of a Kuznetsov category can be given for the 5-tuple, without the need for any conditions.   

\begin{definition} \label{defn:glsm}
    A \newterm{gauged linear sigma model (GLSM)} is the data $(V, \Gamma, \chi, \theta, w)$ consisting of:
    \begin{itemize}
        \item a complex vector space $V$;
       \item a reductive algebraic group $\Gamma \subset \text{GL}(V)$ acting linearly on $V$;
       \item a surjective character $\chi: \Gamma \to \mathbb C^*$;
        \item a character $\theta: \ker \chi  \rightarrow \C^*$ such that the semistable locus with respect to $\theta$ is $\Gamma$-invariant; and
         \item a $\chi$-\newterm{semi-invariant superpotential} $w: V \rightarrow \A^1$, that is, a regular map $w$ satisfying $w(g \cdot x) = \chi(g)w(x)$, for all $x \in V$ and $g \in \Gamma$.
    \end{itemize}
\end{definition}

\begin{remark}
        The semi-invariance of $w$ means that it defines a section of the equivariant line bundle $\O(\chi)$ on the quotient stack $[X/G]$, i.e., $w\in \Gamma(X, \O(\chi))^G$.
\end{remark}

\begin{definition}
We say a GLSM is \newterm{abelian} if $\Gamma$ is an abelian group. 
\end{definition}

 \begin{remark}If a GLSM is abelian, then since $\Gamma$ is reductive, one can choose a basis for $V$ so that $\Gamma \subseteq (\Gm)^{\n}$. Consequently, abelian GLSMs are amenable to toric geometry.\end{remark}

We recall the setup of geometric invariant theory (GIT). Let $G$ be a group acting on a vector space $V$. A GIT quotient $V\modmod{\theta} G$ depends on a choice of character $\theta: G \rightarrow \Gm$; it is the closed $G$-orbits acting on the $\theta$-semistable locus, $V_{\theta}$. This also gives us a stacky GIT quotient $[V\modmod{\theta}G] := [V_\theta/G]$.  The variety $V\modmod{\theta}G$ is isomorphic to the stack $[V\modmod{\theta}G]$ when $G$ acts freely on $V_\theta$.

Define $\widehat{G}_{\Q} := \Hom(G, \Gm) \otimes_{\Z} \Q$. The different semistable loci partition $\widehat{G}_{\Q}$ into subsets
$\{ \tau' \in \widehat{G}_{\Q} \mid V_{\tau} = V_{\tau'} \}.$
If a point $p \in V$ is both $\tau_1$-semistable and $\tau_2$-semistable then it is $(n_1 \tau_1 + n_2 \tau_2)$-semistable for all $n_1, n_2 > 0$. Therefore, we define 
\[
\sigma_\tau:= \overline{\{ \tau' \in \widehat{G}_{\Q} \mid V_{\tau} = V_{\tau'} \}},
\]
which is a (rational) cone. The set of all such cones form a fan called the \newterm{GIT fan} (see \cite{Ress00}).  In the abelian case, this fan is also called the GKZ fan or secondary fan. The maximal cones of this fan are called \newterm{chambers} and the codimension one cones \newterm{walls}.

\begin{remark} \label{rem:stability_vs_equivariance}
    We define the semistable locus $V_{\theta}$ using a character of the \emph{subgroup} $\ker \chi \subset \Gamma$, while requiring that the factorization category is equivariant with respect to the \emph{entire group} $\Gamma$. This holds, for instance, if the GLSM is abelian, geometric (see Definition~\ref{def: geometric}), or has a good lift in the sense of \cite[Section 1.1]{FJR18}.
    %This also holds if $V_{\theta_{K + \epsilon}}$ is also the semistable locus for some character $\psi$ of $\Gamma$. Therefore, the original mathematical definition of a GLSM in \cite[Section 1.1]{FJR18} requires that $\theta_{K + \epsilon}$ be the restriction of a character of $\Gamma$, i.e., a \emph{good lift} of $\theta_{K + \epsilon}$. %For abelian GSLM, the existence of a (not-necessarily good) lift is automatic. To see this notice that the (contravariant) character functor $\widehat{( - )}$ is exact on diagonalizable groups by Pontryagin duality. So applying it to the short exact sequence $1 \rightarrow \ker(\chi) \rightarrow \Gamma \rightarrow \C^* \rightarrow 1$ gives a surjection $\widehat{\Gamma} \rightarrow \widehat{\ker(\chi)}$, implying that every character of $\ker(\chi)$ has a lift from $\Gamma$. 
%    After specializing to a geometric GLSM below, the existence of a good lift is automatic. \dan{I added some commentary on good lifts.}

\end{remark}

We use GIT in the GLSM context setting $G = \ker(\chi)$. The inclusion
\[
\iota: \ker(\chi) \hookrightarrow \Gamma \hookrightarrow \text{GL}(V)
\]
equips $\ker(\chi)$ with the canonical character
\begin{equation}\label{eqref: canonical character}
\theta_K : \ker(\chi) \hookrightarrow \text{GL}(V) \rightarrow \C^* \hspace{1cm} g \mapsto \det(\iota(g))^{-1}.
\end{equation}
Let $\theta_{-K} := \theta_K^{-1}$, denote the anti-canonical character.

\begin{definition}\label{def: Kuznetsov chamber}
A \newterm{Kuznetsov chamber} is a chamber $\sigma_K$ which contains $\theta_K$.  An \newterm{anti-Kuznetsov chamber} is a chamber $\sigma_{-K}$ which contains $\theta_{-K}$. 
\end{definition}
% \noindent We write $U_{\sigma_{K}}$ (respectively $U_{\sigma_{-K}}$) for the semistable locus $(V)_{\theta_K}^{\op{ss}}$ (respectively $(V)_{\theta_{-K}}^{\op{ss}}$) corresponding to the Kuznetsov chamber $\sigma_K$ (anti-Kuznestov chamber $\sigma_{-K}$).

\begin{definition}\label{def: Kuznetsov category}
    Given a GLSM $\mathcal{G} := (V, \Gamma, \chi, \theta, w)$ we say a factorization category
    $$
    \mathcal K := \dabs(V_{\theta_{K+\epsilon}}, \Gamma, w)
    $$
    is a \newterm{Kuznetsov category} for $\mathcal G$ if $\theta_{K+\epsilon} \in \widehat{\ker(\chi)}_{\Q}$ lies in the interior of a chamber containing $\theta_K$. Analogously, we say a factorization category
    $$
    -\mathcal K := \dabs(V_{\theta_{-K-\epsilon}}, \Gamma, w)
    $$
    is an \newterm{anti-Kuznetsov category} for $\mathcal G$ if $\theta_{-K-\epsilon}$ lies in the interior of a chamber containing $\theta_{-K}$.
\end{definition}

\begin{remark}
By definition, a Kuznetsov category is the absolute derived category of the GLSM $(V, \Gamma, \chi, \theta_{K+\epsilon}, w)$. This transparently does not depends on $\theta$ from the original GLSM.  On the other hand, it apriori depends on the choice of $\theta_{K+\epsilon}$ or, equivalently, the choice of Kuznetsov chamber. 
\end{remark}

If $\sigma_K$ is in a chamber of the GIT fan, then the Kuznetsov chamber is unique. However, if $\theta_K$ is on a wall (or a higher codimension cone), then there can be several Kuznetsov chambers. 
The following theorem \cite[Theorem 5.2.1]{BFKv2} allows us to prove that their corresponding Kuznetsov categories are equivalent in the abelian case.

\begin{theorem}\label{thm: BFK wall crossing} 
Let $(V, \Gamma, \chi, \theta_+, w)$ and $(V, \Gamma, \chi, \theta_-, w)$ be abelian GLSMs so that $\theta_+, \theta_-$ lie in adjacent chambers sharing a common wall defined by a primitive one parameter subgroup $\lambda$ of $\ker \chi$ normalized so that $\theta_+ \circ \lambda(t) = t^l$ and $l >0$.
\begin{enumerate}
\item If $\theta_K \circ \lambda (t)=1$, there is an equivalence
$$
 \dabs( V_{\theta_-}, \Gamma, w ) \cong \dabs(V_{\theta_+}, \Gamma, w).
$$
 \item If $\theta_K \circ \lambda (t) = t^r$  and $r>0$, there is a semi-orthogonal decomposition
$$
 \dabs( V_{\theta_-}, \Gamma, w ) = \langle \mathcal W_1, ..., \mathcal W_r, \dabs( V_{\theta_+}, \Gamma, w ) \rangle,
$$
where $\mathcal W_i \cong \dabs((V^\lambda)_{\theta_W}), \Gamma/\lambda, w|_{V^\lambda})$ for all $i$ where $V^\lambda$ is the fixed locus of $\lambda$, $\theta_W$ is any character in the relative interior of the wall descended to $\Gamma/\lambda$.
\end{enumerate}
\end{theorem}

\begin{corollary} \label{cor: independence}
If $\mathcal{G}$ is abelian GLSM, then the Kuznetsov category $\mathcal{K}$ is independent of the choice of Kuznetsov chamber $\sigma_K$. Similarly, the anti-Kuznetsov category $-\mathcal{K}$ is independent of the choice of anti-Kuznetsov chamber $\sigma_{-K}$.
\end{corollary}
\begin{proof}
The proof is the same for the Kuznetsov and anti-Kuznetsov category.  We provide the proof for the Kuznetsov category.

Take two chambers $\sigma_K, \sigma_K'$ containing $\theta_K$ in $\widehat{\ker(\chi)}_{\R}$ and consider a sufficiently small ball $B$ around $\theta_K$.  This ball contains a point in all possible Kuznetsov chambers in particular points $\theta_{K+ \epsilon} \in \sigma_K, \theta_{K+\epsilon'} \in \sigma_K'$.   Now choose a path $\gamma$  in $B$ from $\theta_{K+ \epsilon}$ to $\theta_{K+\epsilon'}$ which does not pass through any codimension $m$ cones for $m \geq 2$.  Since $B$ is sufficiently small, every wall that $\gamma$ passes through will contain $\theta_K$.  Every time we cross a wall, we use  Theorem~\ref{thm: BFK wall crossing}(i), which says that the corresponding absolute derived categories are equivalent, as desired.
\end{proof}

\begin{remark}
Corollary~\ref{cor: independence} holds for some non-abelian $G$ under additional hypotheses. For example for a Calabi-Yau GLSM, $\theta_K$ is the origin and $\sigma_K$ is any chamber.  Then, we expect all chambers to be equivalent.  This has been established in the quasi-symmetric case \cite[Corollary 5.2]{HLS} and for certain examples (e.g., \cite{KimoiEd}).  
\end{remark}

We are particularly interested in GLSMs coming from classical algebraic geometry.  These are defined as follows.

\begin{definition}\label{def: geometric}
A GLSM $(V, \Gamma, \chi, \theta, w)$ is called \newterm{geometric} if the following conditions hold:
    \begin{enumerate}
        \item $V\modmod{\theta} \ker \chi$ is the total space of a vector bundle $\mathcal E$ over a space $X$;
        \item $w$ is a pairing with a regular global section $s$ of $\mathcal E^\vee$; and 
        \item $\chi$ splits so that $\Gamma \cong \ker \chi \times \mathbb G_m$ and the $\Gm$ factor acts by fiberwise dilation on $\mathcal E$.
    \end{enumerate}
\end{definition}
 The term geometric comes from high-energy theoretical physics.  It refers to the fact that a geometric GLSM $(V, \Gamma, \chi, \theta, w)$ admits the same conformal field theory as the complete intersection $Z(s) \subseteq X$.  This is realized mathematically by the following result \cite{Isik, Shipman, Hirano}. 

\begin{theorem}[Proposition 4.8 of \cite{Hirano}] \label{thm: IsikShipmanHirano}
If $(V, \Gamma, \chi, \theta, w)$ is geometric, then there is an equivalence of categories
\[
\dbcoh{Z(s)} \cong  \dabs(V_\theta, \Gamma, w).
\]
\end{theorem} 

By comparing certain algebraic varieties and GLSMs, we obtain Kuznetsov components of derived categories of certain algebraic varieties as Kuznetsov categories of GLSMs.
We give a simple motivating example.
\begin{example}\label{standard example}
Consider the GLSM $(V, \Gamma, \chi, \theta, w)$ where $V = \C^{n+1}$ and $\Gamma =  \mathbb G_m \times \mathbb G_m$ acts with weights $1, \dots,  1, -d$ for the first factor and weights $0,\dots, 0, 1$ for the second factor. Let $\chi = \pi_2$ and $\theta= \pi_1$. Let  $w=x_{n+1}s(x_1, \dots, x_n)$, where $s$ is a homogeneous polynomial of degree $d$, so $w$ is homogeneous of degree zero. 

This GLSM is geometric, corresponding to a hypersurface $Z(s)$ of degree $d$ in $\mathbb P^{n-1}$. By Theorem~\ref{thm: IsikShipmanHirano}, we have the equivalence $\dbcoh{Z(s)} \cong  \dabs(V_\theta, \Gamma, w)$. 

We compute that $\theta_K = \pi_1^{-n+d-1}$. Depending on $n$ and $d$, the character $\theta_K$ can lie in any of the three cones in the GIT fan, with the positive chamber corresponding to the original geometric phase.  These three cases correspond exactly to the semi-orthogonal decompositions described by Orlov \cite{Orlov}.  Namely, by Theorem~\ref{thm: BFK wall crossing}, there is a semi-orthogonal decomposition
        $$
 -   \mathcal K \cong \langle \mathcal{W}_1, \dots, \mathcal{W}_{|n+1-d|}, \mathcal K\rangle, 
    $$
    where $\mathcal{W}_i \cong \dabs(\op{Spec \C}, \gm, 0) \cong \dbcoh{ \op{Spec \C}} \cong \op{D^b(\mathbb C\op{-vect})}$ and:
\begin{enumerate}
    \item If $n+1 = d$, then $\theta_K$ is on the wall so $$\mathcal K \cong -\mathcal K \cong \dabs (V_{\theta}, \Gamma, w)\cong \dabs (V_{\theta^{-1}}, \Gamma, w);$$
    \item If $n+1>d$, then $\theta_K$  is in the positive chamber so $$\mathcal K \cong \dabs (V_{\theta^{-1}}, \Gamma, w), \qquad -\mathcal K \cong   \dbcoh{ Z(s)} \cong \dabs (V_\theta, \Gamma, w);$$
    \item If $n+1 < d$, then $\theta_K$ is in the negative chamber so $$\mathcal K \cong \dabs (V_{\theta}, \Gamma, w), \qquad -\mathcal K \cong \dabs (V_{\theta^{-1}}, \Gamma, w) \cong \dbcoh{ Z(s)}.$$
\end{enumerate}

%by \cite[Proposition 2.1.6]{HPD}. 
The explicit matching of the GLSM story with Orlov's result was originally examined by Segal \cite{Segal} with follow-up works by Shipman \cite{Shipman} and Ballard-Favero-Katzarkov \cite{VGIT}. In the case where $n+1 > d$, Kuznetsov \cite[Corollary 4.4]{Kuz05} shows that $\mathcal{K}$ is a fractional Calabi-Yau of dimension $(n+1)(d-2)/d$. This also follows from Orlov's description \cite{Orlov} which also shows that if $n+1 < d$ the category $-\mathcal K$ is fractional Calabi-Yau.  Furthermore, if $d$ divides $n+1$, or $d$ is even and $d/2$ divides $n+1$, then $\mathcal{K}$ is Calabi-Yau. We remark that $\mathcal K$ in this example is often referred to as the Kuznetsov or residual component of $\dbcoh{Z(s)}$.
\end{example}

 \section{Kuznetsov Categories for complete intersections in toric varieties}\label{sec:toric complete intersections}

Let $V$ be an $n$-dimensional $\mathbb C$-vector space and assume that $G$ is abelian, or, equivalently, $G \subseteq (\C^*)^{\n} \subseteq \text{GL}(V)$.
In this case, the combinatorics of the GIT fan is completely explicit \cite{GKZ}.  In light of this work, it is most often referred to as the GKZ fan. % and denoted $\Sigma_{\op{GKZ}}$. 

Here, every GIT quotient $V\modmod{\theta} G$ with $\theta$ lying in the interior of the support of the GKZ fan corresponds to a toric variety  $X_\Sigma$. Using the inclusion $\iota: G \hookrightarrow (\C^*)^{\n}$, we obtain an exact sequence 
\begin{equation}\label{ses: GIT}
0 \to M \xrightarrow{\nu} \widehat{(\C^*)^{\n}} \xrightarrow{\hat\iota} \widehat G \to 0
\end{equation}
where $M$ is the kernel of $\hat \iota$.  The dual map $\nu^*: \mathbb Z^{\n} \to \Hom(M, \mathbb Z)$ gives a collection of lattice points
\[
\boldsymbol{\nu} :=\{ \nu^*(e_i) \} \subset \Hom(M, \mathbb Z).
\]

The GIT quotients which occur are precisely those corresponding to semiprojective toric varieties $X_\Sigma$ with $|\Sigma| = \text{Cone}(\boldsymbol{\nu})$ and $\Sigma(1) \subseteq \boldsymbol{\nu}$.  In our notation this means, $V_{\theta} = \A^{\Sigma(1)} \setminus Z(B(\Sigma)) \times (\mathbb C^*)^{\nu \setminus \Sigma(1) }$ where $B(\Sigma)$ is the irrelevant ideal for $\Sigma$.

Now, suppose $X$ is a projective toric variety. Then  we can write $X= V \modmod{\theta} G$ for some $G \subseteq (\C^*)^{\n}$ and a character $\theta: G \to \Gm$.  In fact, 
\[
G = \Hom(\op{Cl}(X), \mathbb G_m)
\]
is the character group of the class group of $X$ and the character $\theta \in \widehat G = \op{Cl}(X)$ is ample. We denote by $\theta_{D}$ the character in $\widehat{G}$ when we regard a divisor $D \in \op{Cl}(X)$ as a character.

Now consider a complete intersection in $X$ defined by the common zero locus of global sections $s_i \in \Gamma(X, \mathcal O_X(D_i))$ for $1 \le i \le r$.  
Take the vector bundle $\mathcal E = \bigoplus_{i=1}^r \O_{X}(-D_i)$.  We denote the total space by $\tot(\mathcal E) := \op{Spec} \op{Sym} \mathcal E^\vee = \op{tot} \bigoplus_{i=1}^r \O_{X}(-D_i)$.  To ensure that $\tot(\mathcal E)$ is also a GIT quotient, we assume the divisors $D_i$ are nef (see e.g.\ \cite[Lemma 5.19]{FK17} following \cite[Chapters 6 and 7]{CLS}).

This allows us to construct a geometric abelian GLSM. Namely,  we take $\Gamma = G \times \mathbb G_m$ and consider the $\Gamma$-action on $V \times \mathbb C^r$ given by 
\begin{equation}\label{eqn: G action on bundle}
(g, h) \cdot(v, t_1, ..., t_r) = (g \cdot v,  \theta_{-D_1}(g)ht_1, ..., \theta_{-D_r}(g)ht_r).
\end{equation}
Then, $V \times \mathbb C^r \modmod{\theta} G = \tot(\mathcal E)$ and the additional $\mathbb G_m$ acts by fiberwise dilation.  The action restricted to $G$ in \eqref{eqn: G action on bundle} provides us with an inclusion $\iota': G \to (\C^*)^{\n + r}$. We define the map $\nu': M\oplus \Z^r \to \widehat{(\C^*)^{\n+r}}$ to obtain an analogous short exact sequence to ~\eqref{ses: GIT}.

%Here is the earlier text commented out
% % Let $\Sigma$ be a fan with collection of rays $\Sigma(1)$. Take $\nu = \nu_\Sigma := \{ u_i \ | \ \rho_i \in  \Sigma(1)\}$ to be a collection of primitive lattice generators for the rays. Then we can define the \newterm{Cox fan} associated to $\Sigma$ to be 
% % \[
% % \op{Cox}(\Sigma) := \{ \op{Cone}(e_\rho \ | \ \rho \in \sigma) \ | \ \sigma \in \Sigma\}. 
% % \]
% % This is a subfan of the standard fan for $\A^n$, and thus defines an affine toric variety $U_{\Sigma} = X_{\op{Cox}(\Sigma)}$ which we call the \newterm{Cox open set} associated to $\Sigma$. We then define the \newterm{Cox stack} associated to $\Sigma$ to be 
% % \begin{equation}\label{def: Cox Stack}
% % \mathcal{X}_{\Sigma} := [U_{\Sigma}/ S_{\nu_{\Sigma}}].
% % \end{equation}

%  By \cite[Lemma 4.14]{FK18} and \cite[Lemma 4.15]{FK18}, we have the isomorphism $S_{\Sigma_{\mathbf{D}}} \cong S_{\Sigma}$ and an equality $U_{\Sigma_{\mathbf{D}}} = U_{\Sigma} \times \A^r$. Hence
% \[
% \op{tot} \bigoplus_{i=1}^r \O_{\mathcal{X}_{\Sigma}}(-D_i) = [U_{\Sigma_{\mathbf{D}}} / S_{\Sigma}]
% .\] 
% Note that since $\tot(\mathcal E)$ is a semiprojective toric variety, there exists a character $\theta$ so that 
% \[
% \mathcal{X}_{\Sigma_{\mathbf{D}}} = [ \A^{|\Sigma(1)|+r} \modmod{\theta} S_{\Sigma}].
% \]

The dilation action scaling the fibers of the total space extends to the coordinate ring.  Namely, write $u_1, \dots, u_r$ for the coordinates of $\mathbb C^r$. The global function 
$$
w = \sum_i u_if_i: U_{\Sigma} \times \A^r \to \A^1
$$
is then $\chi$-semi-invariant, where $\chi : \Gamma \to \Gm$ is the character defined by the projection to the second factor. This gives the GLSM 
$$
(V \times \mathbb C^r, \Gamma, \chi, \theta, w)
$$
in the notation of Definition~\ref{defn:glsm}.

For convenience, we provide a description of the fan for $\tot(\mathcal E)$ here.  We denote the fan for $\tot(\mathcal E)$ by $\Sigma^{\mathcal{E}}$.
We recall the description of $\Sigma^{\mathcal{E}}$ (from, e.g., \cite[\textsection 7.3.1]{CLS}), built from the fan $\Sigma^X$ for $X$.  The fan $\Sigma^{\mathcal{E}}$ is contained in $N_{\R} \otimes \R^r$. Write $e_i$ for the standard basis of $\R^r$ and $\tau_i$ for the ray generated by $e_i$, and express $D_i = \sum_{\rho\in \Sigma^X\!(1)} a_{i\rho}D_\rho$ as a sum of torus-invariant Weil divisors. From the original rays of $\Sigma^X$, we obtain new rays $\bar \rho := \op{Cone}(u_\rho + \sum_{i=1}^r a_{i\rho}e_i)$ for all $\rho \in \Sigma^X(1)$. The rays of $\Sigma^{\mathcal{E}}$ are then the set 
\[
\boldsymbol{\nu}' := \Sigma^{\mathcal E}(1) = \{ \bar \rho \ | \ \rho \in \Sigma^X(1)\} \cup \{\tau_1, \dots, \tau_r\}.
\]
The cones in the fan $\Sigma^{\mathcal{E}}$ consist of cones of the form
$$
\sigma_{D, I} := \op{Cone}(\bar \rho, \tau_i \ | \ \rho \in \sigma(1), i \in I)
$$
for any $\sigma \in \Sigma^X$, $I \subseteq \{1, \dots, r\}$. 

\begin{remark}
To convert between the GIT/GLSM descriptions and the toric ones, observe that in a basis the matrix $\nu$ (resp.\ $\nu'$) has rows given by the elements of $\boldsymbol{\nu}$ (resp.\ $\boldsymbol{\nu}'$).
\end{remark}

Consider the GIT problem with the group $G$ acting on the vector space $V \times \C^r$ as above.
Then \cite[Lemma 4.14]{FK18} implies that the following diagram has exact rows and commutes
%but with the conclusion that $S_{\Sigma_{\mathbf{D}}} \cong S_{\Sigma}$ substituted in:
\begin{equation}
\xymatrix@C=1.3em{
0 \ar[rr] && M\oplus\Z^r \ar[rr]^-{\nu'} \ar[d]_{\text{proj}_M} && \Z^{\n+r} \ar[rr]^-{\hat \iota'} \ar[d]_g &&  \widehat{G} \ar[d]_{\mathrm{id}} \ar[rr] && 0\\
0 \ar[rr] && M  \ar[rr]^-{\nu}  &&  \Z^{\n} \ar[rr]^-{\hat{\iota}} && 
\widehat{G}  \ar[rr] && 0,
}
\label{eq: projection diagram}
\end{equation}
where $\op{proj}_M: M\oplus\Z^r \rightarrow M$ is the standard projection and $g$ is defined by
\begin{equation} \label{eqdef: g}
g: \Z^{\n+r} \rightarrow \Z^{\n},  \ \ e_\rho \mapsto e_\rho, \ \  e_i \mapsto -\sum_\rho a_{i\rho} e_\rho.
\end{equation}
By chasing the diagram, one sees that $\hat\iota(e_i) = -\sum_{\rho} a_{i\rho} D_\rho$. 

Each row of \eqref{eq: projection diagram} gives rise to a toric GIT problem. We denote the corresponding GKZ fans for each short exact sequence by $\Sigma^{\mathcal E}_{\op{GKZ}}$ and $\Sigma^X_{\op{GKZ}}$. Both are fans in $\widehat G_{\R}$. By \cite[Theorem 14.4.7]{CLS}, the support of these fans are as follows:
\begin{align*}
    |\Sigma^{\mathcal E}_{\op{GKZ}}| &= \op{Cone}(\hat \iota'(e_i), \hat \iota'(e_\rho) \ | \ i\in \{1, \dots, r\}, \rho \in \Sigma(1)); \\ |\Sigma^X_{\op{GKZ}}| &= \op{Cone}(\hat\iota(e_\rho) \ | \ \rho \in \Sigma(1)).
\end{align*}
Thus $|\Sigma^X_{\op{GKZ}}| \subseteq  |\Sigma^{\mathcal E}_{\op{GKZ}}|$, by the commutativity of \eqref{eq: projection diagram} and the fact that $g(e_\rho) = e_\rho$. Moreover, since $|\Sigma^X|=\R^n$, Gale duality between ray generators and torus-invariant divisor classes implies that the cone $|\Sigma^X_{\op{GKZ}}|$ is strongly convex, see \cite[Lemma 14.3.2]{CLS}. 

On the other hand, if any of the $D_i$ lie in the interior of the effective cone (e.g., some $D_i$ is ample) then $|\Sigma^{\mathcal E}_{\op{GKZ}}| = \widehat G_{\mathbb R}$ since this is the cone generated by a full dimensional strongly convex cone $|\Sigma^X_{\op{GKZ}}|$ and at least one vector is in the negative of the interior.

Let $W$ be a wall in $\Sigma^{\mathcal E}_{\op{GKZ}}$ and let $V_W$ denote its linear span in $\widehat G_{\mathbb R}$. Define
\begin{align*}
\boldsymbol{\beta}^X_W &:= \{\beta_\rho:=\hat\iota(e_\rho) \ | \ \rho \in \Sigma_{\op{GKZ}}(1), \pi(e_\rho) \in V_W \}; \\
\boldsymbol{\beta}^{\mathcal{E}}_W &:= \boldsymbol{\beta}^X_W\cup  \{\beta_i:= \hat\iota'(e_i) \ | \ i \in \{1, \dots, r\}, \pi(e_i) \in V_W\};
\end{align*}
and consider the following cones in $\widehat G_{\mathbb R}$
\begin{align*}
C_W^X & := \op{Cone}(\boldsymbol{\beta}^X_W); \\
C_W^{\mathcal{E}} & := \op{Cone} (\boldsymbol{\beta}^{\mathcal{E}}_W ). 
\end{align*}
\noindent The cone $C^X_W$ is a subcone of the strongly convex cone $|\Sigma^X_{\op{GKZ}}|$ hence is also strongly convex. We note that $C^X_W$ need not be a cone in $\Sigma^X_{\op{GKZ}}$.

\begin{lemma} \label{prop: ample cone}
    The cone $C_W^X$ and the ample cone $\Amp(X)$ of $X$ are disjoint.
  %  Let $\Amp(X)$ denote the ample cone of $X$. Let $C^X_W$ denote a cone generated by weights of $X$ in $V_W$. Then $C^X_W \cap \Amp(X) = \emptyset$.
 \end{lemma}
 
 \begin{proof}
     Suppose that $C^X_W \cap \Amp(X) \ne \emptyset$. %Since every cone has a simplicial refinement....
    Then, there exists a linearly independent set $J \subset \{ \rho \in \Sigma^X_{\op{GKZ}}(1), D_\rho \in V_W\}$ so that the cone $C_J := \op{Cone}(\beta_\rho \ | \ \rho \in J)$ contains a point $\theta$ in the ample cone.  Since this cone is not full dimensional, we extend $\{\beta_\rho\}_{\rho \in J}$ to a $\beta$-basis $\{\beta_\rho\}_{\rho \in I}$ and consider the cone $C_\beta := \op{Cone}(\{\beta_\rho\}_{\rho \in I})$. 

    Since $ \op{Nef}(X) \subseteq C_\beta$ by \cite[Proposition 15.2.1]{CLS} and both are full-dimensional cones, we have that  $\op{Amp}(X) = \op{int}\op{Nef}(X) \subseteq \op{int} C_\beta$. However, $\theta$ is on a face of the cone $C_\beta$ as it is in the proper face $C_J$ of the cone $C_\beta$.
 \end{proof}

\begin{proposition}\label{prop:strongly convex weight cone}
Suppose that the divisors $D_i$ are ample for all $i$. For any wall $W \in \Sigma^{\mathcal E}_{\op{GKZ}}$, the cone $C^{\mathcal E}_W$ is strongly convex.
\end{proposition}

\begin{proof}
By Lemma~\ref{prop: ample cone} and the hyperplane separation theorem for convex sets (see e.g. \cite[3.4]{Klee}), there exists a hyperplane\footnote{Technically, the result cited only gives the existence of an affine hyperplane. This hyperplane must pass through the origin since both convex sets contain the origin in their closure.}
$H$ that separates $C^X_W$ and $\Amp(X)$. So $C^X_W$ and $-\Amp(X)$ lie on the same side of $H$, which we choose to be the positive side of $H$. Using the commutativity of \eqref{eq: projection diagram}, we see that $\beta_i \in - \Amp(X)$, hence $C^{\mathcal E}_W = \Cone(C^X_W, \boldsymbol{\beta}^{\mathcal{E}}_W \setminus \boldsymbol{\beta}^X_W)$ in $V$ lies on the positive side of $H$. % The cone $C^{\mathcal E}_W := \Cone(C^X_W, D_i)$ over all $D_i$ in $V$ lies on the positive side of $H$. \tyler{There needs to be some justification for the fact that $D_i \in - \op{Amp}(X)$, which implies this claim.}

Now, notice $C^{\mathcal E}_W \cap -C^{\mathcal E}_W$ is contained in $H$. But $H \cap C^{\mathcal E}_W = H \cap C^X_W$, as each $-D_i$ is strictly positive for $H$. Since $C^X_W$ is a strongly convex cone, we have 
\begin{align*}
C^{\mathcal E}_W \cap -C^{\mathcal E}_W &\subseteq H \cap C^{\mathcal E}_W \cap -C^{\mathcal E}_W \\
& = (H \cap C^{\mathcal E}_W) \cap (H \cap -C^{\mathcal E}_W) \\
& = (H \cap C^X_W) \cap (H \cap - C^X_W) \\ 
& = \{0\}.
\end{align*}
\end{proof}

%Consider the global function $w_W$ on $[\A^{J_W} \modmod{\chi_W} S_W \times \Gm]$ as in \eqref{eq: stack for the wall}.

\begin{lemma} \label{lem: w=0}
If the divisors $D_i$ are ample for all $i$, then the function $w|_{(V \times \mathbb C^r)^\lambda}=0$. 
\end{lemma}

\begin{proof}  The function $w \in S = \op{Sym} (V \times \mathbb C^r)^\vee$ has degree $(0,1)\in \widehat{ G \times \Gm}$.  Its restriction to $(V \times \mathbb C^r)^\lambda$ is an element of degree $(0,1) \in \widehat{G/\lambda \times \Gm}$ in $\op{Sym}  ((V \times \mathbb C^r)^\lambda)^\vee$.  The polynomial ring $\op{Sym} ((V \times \mathbb C^r)^\lambda)^\vee$ has coordinates with degrees lying in $\widehat{G/\lambda} \cap C^\mathcal E_W$.  Thus, since $C^\mathcal E_W$ is strongly convex by Proposition~\ref{prop:strongly convex weight cone}, $w|_{(V \times \mathbb C^r)^\lambda}$ is constant.  However, $w|_{(V \times \mathbb C^r)^\lambda}$ has non-trivial degree in $\widehat \Gm$ and is therefore $0$.
% Take any monomial $m$ in $w$. It can be written in the form
% $$
% u_i \prod_{\rho \in R} x_\rho^{m_{i\rho}}
% $$
% for some $c \in \Gm$, $i \in \{1, \dots, r\}$, $R \subseteq \Sigma(1)$, and $m_{i\rho} \in \Z_{>0}$. If $\pi(e_i)$ or $\pi(e_\rho)$ for any $\rho \in R$ is not in the linear span $V_W$ of the wall $W$, then the monomial vanishes in $w_W$. 
% Since $w$ is homogeneous of degree $(0,1)$ in $S_{\Sigma} \times \Gm$, it is also homogeneous of degree $(0,1)$ in $S_W \times \Gm$. Thus $m$ is as well. However, the weights form a strongly convex cone $C_W^{\mathcal{E}}$ in $V_W$, by Proposition~\ref{prop:strongly convex weight cone}, hence either $\pi(e_i) \notin V_W$ or  there must be a $\rho \in R$ with $\pi(e_\rho) \notin V_W$, hence $m|_{\A^{J_W}} = 0$. This proves the claim.
\end{proof}

\begin{theorem}\label{thm: main theorem}
Let $Z$ be a complete intersection of ample divisors in a smooth projective toric variety (or DM stack).  Then, there are semiorthogonal decompositions
\[
\dbcoh{Z} = \langle \mathcal K, E_1, \dots, E_t \rangle
\]
and
\[
-\mathcal K = \langle \dbcoh{Z}, E_1, \dots, E_s \rangle\]
where the $E_i$ are exceptional objects.
\end{theorem}

\begin{proof}
We give the proof of the first semiorthogonal decomposition.  The proof for the second is exactly the same with $K$ replaced by $-K$.

Choose a generic $\theta \in \Amp(X)$ so that the straight line path $\gamma: [0,1] \to \widehat G_{\mathbb{R}}$ such that $\gamma(0) =\theta$ to $\gamma(1) =\theta_{K}$ and $\gamma([0,1))$ does not intersect any codimension $m$ cones for $m \geq 2$. 
This path crosses through a finite number of walls and, by construction, $\theta_K$ lies on the positive side of the hyperplane spanned by each wall in the direction of the path.  Therefore, by Theorem~\ref{thm: BFK wall crossing}, at each wall-crossing, we pick up some number of copies of $\mathcal W_i$.

Now, 
\begin{align*}
\mathcal W_i & = \dabs((V^\lambda)_{\theta_W}, \Gamma/\lambda, w|_{V^\lambda})  & \\
& = \dabs((V^\lambda)_{\theta_W}, \Gamma/\lambda, 0) )& \text{by Lemma~\ref{lem: w=0}} \\
&= \dbcoh{[(V^\lambda)_{\theta_W} / (\ker \chi /\lambda)]} &\text{by \cite[Proposition 2.1.6]{HPD}}
\end{align*}

The support of the fan for $[(V^\lambda)_{\theta_W} / (\ker \chi /\lambda)]$ is Gale dual to $C_W^{\mathcal E}$, hence fan is complete since Proposition~\ref{prop:strongly convex weight cone}, $C_W^{\mathcal E}$ is strongly convex. It follows that $[(V^\lambda)_{\theta_W} / (\ker \chi /\lambda)]$ is proper. Thus $[(V^\lambda)_{\theta_W} / (\ker \chi /\lambda)]$ is projective since it is a GIT quotient. Moreover, the fan is simplicial since $\theta_W$ is generic for the GIT problem in the wall. Therefore, $[(V^\lambda)_{\theta_W} / (\ker \chi /\lambda)]$ is a smooth projective DM toric stack.

By \cite[Theorem 1.1]{Kawamata} (see also \cite[5.2.3]{VGIT} for this specific class of stacks), the derived category $\dbcoh{[(V^\lambda)_{\theta_W} / (\ker \chi /\lambda)]}$ admits a full exceptional collection. Now we are done as we have shown that every wall crossing contributes some number of exceptional objects.
\end{proof}

\begin{remark}
    The fully-faithful embeddings of $\mathcal K$ into $\dbcoh{Y}$ and $\dbcoh{Y}$ into $-\mathcal K$ depend on the choices made to produce the semi-orthogonal decompositions in the proof of Theorem~\ref{thm: main theorem}.  Therefore, the number of exceptional objects $t$ and $s$ may depend on these choices as well.  However, we expect that they do not.  This independence would follow from a version of \cite[Theorem A]{KS22} for matrix factorization categories.
\end{remark}

\begin{example}
We demonstrate that the ampleness assumption for the divisors in Theorem \ref{thm: main theorem} is necessary by exhibiting a counterexample under the weaker assumption that the divisors are nef.   
Consider a hypersurface $Y \subseteq \mathbb P^n$ of degree $d\leq n+1$ as in Example~\ref{standard example} so that
\[
\dbcoh{Y} = \langle \mathcal A, \mathcal O_Y(-n+d), ..., \mathcal O_Y \rangle.
\]
The sequence of wall-crossings for $Y \times Y$ as a complete intersection in $\mathbb P^n \times \mathbb P^n$ gives
\begin{align*} \dbcoh{(Y \times Y)} = \langle 
\mathcal A \boxtimes \mathcal A, \mathcal A\boxtimes \mathcal O_Y(-n+d), ..., \mathcal A \boxtimes \mathcal O_Y, \mathcal O_Y(-n+d) \boxtimes \mathcal A, ..., \mathcal O_Y \boxtimes \mathcal A , \\
\mathcal O_{Y \times Y}(-n+d,-n+d), ..., \mathcal O_{Y \times Y} \rangle.
\end{align*}
\end{example}

\begin{example}
Consider a Fano threefold $Y_d$ with $\op{Pic} Y_d = \Z$, index 2, and degree $d$. Then,   there is a semiorthogonal decomposition \cite[Corollary 3.5]{Kuz09}
\[
\dbcoh{Y_d} = \langle \mathcal{B}_{Y_d}, \O_{Y_d}, \O_{Y_d}(H)\rangle.
\]
For $d =1,2, 3, 4$, the category $\mathcal{B}_{Y_d}$ is category $\mathcal K$ in Theorem~\ref{thm: main theorem}. We expect that, for the final Fano threefold $Y_5$ with $\op{Pic} Y_5 = \Z$ and index 2, the category $\mathcal{B}_{Y_d}$ can be realized as the Kuznetsov category of a (non-abelian) GLSM.

As seen in loc. cit., given a Fano threefold $X = X_{2g-2}$ of index 1 and even genus $g = 2t$, then, by using Mukai bundles \cite{Mukai, Kuz09, BKM}, there is a pair $(\mathcal{E}, \O_X)$ of exceptional objects and a semiorthogonal decomposition
$\dbcoh{X_{2g-2}} = \langle \mathcal{A}_{X_{2g-2}}, \mathcal E, \O_{X_{2g-2}}\rangle$. It would be very interesting to naturally realize the category $\mathcal{A}_{X_{2g-2}}$ as a Kuznetsov category of a GLSM, in light of the recent proof by Kuznetsov-Shinder of Kuznetsov's conjecture \cite{KuzShin}.
\end{example}

In the hypersurface case, we can say more.  Suppose $X$ is a projective, smooth, Fano toric variety (or DM stack) and that $D = -qK_{X}$ for some $q \in \Q_{>0}$. Consider the fan $\Sigma^{\mathcal{E}}$ corresponding to the toric stack $\op{tot} \mathcal{E} = \op{tot} \mathcal{O}_{\mathcal{X}_\Sigma}(-D)$. The secondary fan $\Sigma^{\mathcal E}_{\op{GKZ}}$ is a complete fan that contains the secondary fan $\Sigma^X_{\op{GKZ}}$ and an additional ray in the direction of $K_{X} \in \op{Cl}(X)_{\R}$. One computes the canonical character $\theta_K = -(1-q)K_{X}$. Thus if $q <1$, we have that it is on the ray in the secondary fan corresponding to the bundle coordinate. Pick a Kuznetsov chamber $\sigma_K$. We note by Corollary~\ref{cor: independence}, any choice will correspond to an equivalent absolute derived category.  The following is a toric generalization of Example~\ref{standard example}:
\begin{theorem}\label{thm: Fano CY category}
  Suppose $X$ is a projective, smooth Fano toric variety (or DM stack) and $D = -qK_{X}$ for some rational number $q > 0$. Choose a global section $f \in \Gamma(X, \O(D))$. Write $Z$ for the (stacky) toric hypersurface $ Z:=Z(f) \subseteq X$. If $q \leq 1$, then we have a semi-orthogonal decomposition
  $$
  \dbcoh{Z} \cong \langle \mathcal K , E_1, \dots, E_t\rangle,
  $$
  and if $q \geq 1$, then we have a semi-orthogonal decomposition
    $$
  -\mathcal K  \cong \langle\dbcoh{Z}, E_1, \dots, E_s\rangle,
  $$
  where each $E_i$ is an exceptional object.
  If $Z$ is smooth and $q\leq 1$ then $\mathcal K$  is fractional Calabi-Yau.   If $Z$ is smooth and $q\geq 1$ then $-\mathcal K$  is fractional Calabi-Yau.  Furthermore, if $q = \frac{1}{r}$ for some positive $r \in \Z$ and $D$ is Cartier, then $\mathcal K$ is Calabi-Yau. 
\end{theorem}

\begin{proof}
The semi-orthogonal decompositions are an example of Theorem~\ref{thm: main theorem}. %The Kuznetsov category $\mathcal{K}$ in Theorem~\ref{thm: main theorem} corresponds to the GLSM associated with the Kuznetsov chamber $\sigma_K$ in the above discussion. 
%Since $\sigma_K$ is in $-\op{Amp}(X_{\Sigma})$, any cone of a $\beta$-basis in the GKZ fan must include the bundle coordinate, the fan $\Sigma_{\sigma_K}$ will not use the ray associated to the bundle coordinate. We thus use \cite[Corollary 4.23]{FK18} to see that $\mathcal{K} \cong \dabs(U_{\Sigma_{\sigma_K}}, G, f)$.
Now by \cite[Proposition 4.9]{FK18}, $|\Sigma^{\mathcal{E}}|$ is a $\Q$-Gorenstein cone, and if $q = \tfrac{1}{r}$ and $D$ is Cartier then $|\Sigma^{\mathcal{E}}|$ is almost Gorenstein.
By \cite[Lemma 4.24]{FK18}, if $q \leq 1$, then $\sigma_K$ gives a fan whose rays lie at height one of $|\Sigma^{\mathcal{E}}|$.  On the other hand if $q \geq 1$, then $\sigma_{-K}$ gives a fan whose rays lie at height one of $|\Sigma^{\mathcal{E}}|$.
  The claim then follows by applying \cite[Corollary 5.6]{FK18}.
\end{proof}

\begin{remark}
  The existence of these fractional Calabi-Yau categories relating to $\dbcoh{Z}$ is a consequence of \cite[Corollary 5.12]{FK18}; however, the description of the right orthogonal in terms of exceptional collections is new. 
\end{remark}

\section{Anti-Kuznetsov Categories and Fano visitors}\label{sec:FanoVisitors}
Consider a projective variety $X$ obtained as a GIT quotient $X = [V\modmod{\theta} G]$.  
Let $W$ be a $G$-representation.  Let $\mathcal O_X(W)$ be the associated vector bundle and $s \in \Gamma(X, \mathcal O_X(W)) = (\op{Sym}  V^\vee \otimes W^\vee)^G$ be a regular section.
\begin{definition}\label{def:positive throuple}
    Let $V,W$ be $G$-representations and $\theta\in \widehat{G}$. We say $(V,W,\theta)$ is a \newterm{positive triple} if we have the equality of ideals
      \begin{equation}\label{eq: christmas miracle}
    \left\langle(\op{Sym} V^\vee \otimes \op{Sym} W^\vee(\theta^r))^G: r >0\right\rangle =\left\langle(\op{Sym} V^\vee (\theta^r))^G: r>0\right\rangle.
    \end{equation}
\end{definition}

\begin{example}
If $X=[V\modmod{\theta}G]$ is a projective Fano GIT quotient and $W = \bigoplus \mathbb C(-D_i)$ comes from a sum of anti-ample representations,  then $(V, W, \theta)$ is a positive triple (see the proof of Corollary~\ref{cor: xmas miracle}). This is straightforward to see in the toric case.
\end{example}

Then $Z := Z(s)$ is a complete intersection in $X$.   Consider the projectivization, $\mathbb P(\mathcal O_X(W))$. There is a canonical isomorphism 
\begin{equation} \label{eq: can iso}
\Gamma(X, \mathcal O_X(W)) \cong \Gamma(\mathbb P(\mathcal O_X(W)), \mathcal O_{rel}(1)).
\end{equation}
where
$\mathcal O_{rel}(1)$ is the relative hyperplane bundle.
Hence, the section $s$ also defines a hypersurface $Y \subseteq \mathbb P(\mathcal O_X(W))$.
The following theorem generalizes the main result of \cite{KKLL}.

\begin{theorem}  \label{thm: fano visitor}
There is a semi-orthogonal decomposition
\begin{equation}\label{eq: Fano Visitor}
 \dbcoh{Y} = \langle \dbcoh{X}, ..., \dbcoh{X}\otimes \mathcal O_{rel}(\dim W-2), \dbcoh{Z}\rangle.
\end{equation}
Moreover, 
\begin{enumerate}
    \item If $\O_X(\det(W^\vee\oplus V))$ is nef, $\dim W \ge 2$, and $\mathcal O_{rel}(1)$ is ample, then $Y$ is Fano and $Z$ is a Fano visitor.
    \item If $X = [V\modmod{\theta_{-K+\epsilon}} G]$ is sufficiently generic and $(V, W, \theta_{-K+\epsilon})$ is a positive triple, then there is a GLSM such that $\dbcoh{Y}$ is an anti-Kuznetsov category.  
\end{enumerate}
\end{theorem}
\begin{proof}
Consider a new GIT problem of $G \times \Gm$ acting on $V \times W^\vee \times \mathbb C$.  We let $G$ act on $V$ as before, on $W^\vee$ as the dual representation, and trivially on $\mathbb C$.  We let $\mathbb G_m$ act trivially on $V$, by scaling $W^\vee$, and by anti-scaling on $\mathbb C$.  In addition, consider a third $\mathbb G_m$-action given just by scaling the $\mathbb C$ factor.  A choice of any character $\theta$ gives a GLSM $(V \times W \times \mathbb C, G \times \mathbb G_m \times \mathbb G_m, \pi_3, \theta, ps)$ where $\pi_3$ is projection onto the final $\mathbb G_m$ and $p$ is the coordinate on $\mathbb C$.

Consider the 1-parameter subgroup of $\Gamma$ given by the middle $\gm$-factor acting on $V_\theta \times W \times \mathbb C$.
We will use \cite[Theorem 3.5.2]{VGIT} and the notation therein.
%The contracting loci are 
% \begin{align*}
% S_\lambda & = V_\theta \times 0 \times \mathbb C \\
% S_{-\lambda} & = V_\theta \times W \times 0
% \end{align*}
One computes that
\begin{align*}
(V_\theta \times W \times \mathbb C)_+ & = V_\theta \times W \backslash 0 \times \mathbb C \\
(V_\theta \times W \times \mathbb C)_- & = V_\theta \times W \times \mathbb C^*
\end{align*}
so that
\begin{align*}
[(V_\theta \times W \times \mathbb C)_+/\Gamma] & = [\tot(\O_{rel}(-1))/\gm] \\
[(V_\theta \times W \times \mathbb C)_-/\Gamma] & = [\tot(\O_X(W))/\gm]     
\end{align*}
where each $\gm$ acts by fiberwise dilation.
Furthermore
\[
(V_\theta \times W \times \mathbb C)^{\lambda} = V_\theta \times 0 \times 0.
\]
Define
\[
w := s \otimes u \in \op{Sym}V^\vee \otimes\op{Sym}W^\vee \otimes k[u]. 
\]

Hence \cite[Theorem 3.5.2]{VGIT} gives
\begin{align*}
\dabs(X_+, \Gamma, w_+) = \langle \dabs(V_\theta, G \times \gm, 0), \dots, \dabs(V_\theta, G \times \gm, 0) \otimes \mathcal O_{rel}&(\dim W-2), \\ &\dabs(X_-, \Gamma, w_-)\rangle.
\end{align*}
Finally, $\dabs(X_+, \Gamma, w_+) \cong \dbcoh{Y}$ and $\dabs(X_-, \Gamma, w_-) \cong \dbcoh{Z}$ by Theorem~\ref{thm: IsikShipmanHirano} and $\dabs(V_\theta, G \times \mathbb G_m, 0) \cong \dbcoh{X}$ by \cite[Proposition 2.1.6]{HPD}, hence ~\eqref{eq: Fano Visitor} holds.

To prove (i), first note that  since $X$ is a GIT quotient, $\omega_X^{-1} = \O_X(\det(V))$.  Therefore, $\O_X(\det(W^\vee \oplus V)) = (\omega_X \otimes \O_X(W))^{-1}$.  Hence, the assumption that $\O_X(\det(W^\vee \oplus V))$ is nef and $\O_{rel}(1)$ is ample are precisely the assumptions of \cite[Lemma 3.1]{KKLL} which concludes that $Y$ is Fano.

Finally for (ii), choose a character $\theta_{-K + \epsilon} \in \widehat{G \times \gm}$ which is sufficiently generic so that it lies in an anti-Kuznetsov chamber in $(\widehat{G \times \gm})_{\mathbb R}$ and its projection under $\pi: (\widehat{G \times \gm})_{\mathbb R} \to \widehat{G}_{\R}$ lies in an anti-Kuznetsov chamber in $\widehat{G}_{\mathbb R}$.
Then, the GLSM 
$$(V \times W \times \mathbb C, G \times \gm \times \gm, \pi_3, \theta_{-K + \epsilon}, w)$$ is geometric with associated complete intersection $Y$.

To see the final statement, we claim that the semistable locus for $\theta_{-K + \epsilon}$ is $V_{\pi(-K + \epsilon)} \times W \backslash 0 \times \mathbb C$.  Hence, the GIT quotient is the total space of $\mathcal O_{rel}(-1)$ and $w$ descends to the pairing with $s$ via the isomorphism \eqref{eq: can iso}.  This is due to the fact that the ideal defining the unstable locus is, by definition, generated by functions in
\begin{align*}
(\op{Sym}V^\vee \otimes\op{Sym}&W^\vee \otimes k[u](\theta_{-K + \epsilon}^r))^{G \times \gm} \\& = (\op{Sym}V^\vee(\pi(\theta_{-K + \epsilon}^r)))^G\otimes (\op{Sym} W^\vee \otimes k[u](r(\dim W-1)))^{\gm} & \text{by \eqref{eq: christmas miracle}} \\
& = (\op{Sym}V^\vee(\pi(\theta_{-K + \epsilon}^r)))^G\otimes (\op{Sym}^{r(\dim W-1)}W^\vee)
\end{align*}
for $r >0$.  The zero locus is therefore $V_{\pi(\theta_{-K + \epsilon})} \times 0 \times \mathbb C$, as desired. The final statement then follows from Theorem~\ref{thm: IsikShipmanHirano}.
\end{proof}
% By assumption, the anti-canonical character of $G \times \mathbb G_m$ gives an ample divisor on $\mathbb P(\mathcal O_X(W))$  and hence on the total space of the relative tautological bundle $\mathcal O_{rel}(-1)$.    This implies that the GIT quotient for the anti-canonical character is the total space of the relative tautological bundle given by 
% \[
% U := (V \times W \times \mathbb C)_{\theta_{-K' + \epsilon}} = V_{\theta_{-K + \epsilon}} \times (W \backslash  \{\mathbf 0\}) \times \mathbb C.
% \]
% That is, the anti-Kuznetsov category is given by 
% \begin{align*}
% -\mathcal K & = \dabs(U,G \times \mathbb G_m \times \mathbb G_m, ps) & \\
% & = \dbcoh{Y} & \text{by Theorem~\ref{thm: IsikShipmanHirano}}. 
% \end{align*}
% On the other hand, the different choice of stability $\theta_{-K+\epsilon} \cdot \pi_2^{\dim W+2}$ gives 
% \[
% U' := (V \times W \times \mathbb C)_{\theta_{-K+\epsilon} \cdot \pi_2^{\dim W+2}} = V_{\theta_{-K + \epsilon}} \times W \times \mathbb C^*
% \]
% whose associated GIT quotient is the total space of $\mathcal O_X(W^\vee)$ and hence
% \begin{align*}
% \dabs(U',G \times \mathbb G_m \times \mathbb G_m, ps) & \cong  \dbcoh{Z}  & \text{by Theorem~\ref{thm: IsikShipmanHirano}}.
% \end{align*}
% The one-parameter subgroup corresponding to the middle $\mathbb G_m$-factor gives an elementary wall crossing interpolating between these two chambers.  Hence by \cite[Theorem 3.5.2]{BFKv2}, we get a semi-orthogonal decomposition
% \[
% -\mathcal K = \dbcoh{Y} = \langle \dbcoh{X}, ..., \dbcoh{X}\otimes \mathcal O_{rel}(\rank \ W-1), \dbcoh{Z}\rangle.
% \]
%

\begin{remark}
    The proof of~\eqref{eq: Fano Visitor} uses the raw form of \cite[Theorem 3.5.2]{VGIT} which specifies a one-parameter subgroup directly.  It does not use the full GIT fan of the GLSM or address whether or not this comes from a natural wall-crossing from the GLSM. 
\end{remark}

\begin{corollary} \label{cor: xmas miracle}
Any complete intersection of $r\ge 2$ ample divisors in a Fano GIT quotient $X$ is a Fano visitor and the derived category of the Fano host is equivalent to an anti-Kuznetsov category.
\end{corollary}

\begin{proof}
As $X = [V \modmod{\theta}G]$, we have $\omega^{-1}_X \cong \mathcal O_X(\det V)$.  Hence, the assumption that $X$ is Fano is equivalent to requiring that $\mathcal O_X(\det V)$ is ample.  On the other hand, if our complete intersection comes from ample divisors, then $\O_X(W^\vee)$ is a sum of ample line bundles.  Therefore the assumptions of Theorem~\ref{thm: fano visitor}(i) hold.  

Now we check that $(V,W, \pi(\theta_{-K+ \epsilon}))$ is a positive triple.   Consider the trivial subrepresentatations of $\op{Sym }V^\vee \otimes \op{Sym }W^\vee(\theta_{-K+\epsilon}^r)$.  Since $W^\vee$ is a sum of ample representations, so is $\op{Sym }W^\vee(\theta_{-K+\epsilon}^r)$.  Hence, trivial representations can only occur if one cancels a non-trivial representation in $\op{Sym }W^\vee(\theta_{-K+\epsilon}^r)$ with an anti-ample representation in 
$\op{Sym }V^\vee$. This means that any element of $(\op{Sym }V^\vee \otimes \op{Sym }W^\vee(\theta_{-K+\epsilon}^r))^G$ lies in the ideal generated by $$(\op{Sym }V^\vee(\theta_{-K+\epsilon}^r))^G,$$ as desired.
\end{proof}

\begin{remark}
A hypersurface in projective space is a codimension 2 complete intersection of ample divisors in a higher-dimensional projective space. This means that Corollary~\ref{cor: xmas miracle} generalizes the main theorem of~\cite{KKLL}. The authors do not know if a similar construction can be made for hypersurfaces in a general Fano GIT quotient. 
\end{remark}

\bibliography{biblio}
\bibliographystyle{amsalpha}

\end{document}